\documentclass[11pt,a4paper]{article}
  \usepackage{amsmath,amssymb}
  \usepackage{booktabs}
  \usepackage{graphicx}
\newtheorem{lemma}{Lemma}
\newtheorem{definition}{Definition}
\newtheorem{example}{Example}
  \newtheorem{conjecture}{Conjecture}
\newtheorem{corollary}{Corollary}  
\newtheorem{theorem}{Theorem}
  \newcommand{\eop}{\hfill{$\Box$}}
\newenvironment{proof}
{\begin{trivlist}\item[]{{\sc Proof.}}}{\eop\noindent\end{trivlist}}

\begin{document}
  \title{On the Hegselmann-Krause conjecture in opinion dynamics}
  \author{{\sc Sascha Kurz}\thanks{sascha.kurz@uni-bayreuth.de}  \,\,and
      {\sc J\"org Rambau}\thanks{joerg.rambau@uni-bayreuth.de}\\ 
      Department of Mathematics, University of Bayreuth\\ 
      D-95440 Bayreuth, Germany}
  \maketitle
  \vspace*{-4mm}
  \noindent
  {
    \center\small{Keywords: opinion dynamics, consensus/dissent bounded confidence, non linear 
    dynamical systems} 
  }
  \noindent
  \rule{\textwidth}{0.3 mm}

  \begin{abstract}
    \noindent
    We give an elementary proof of a conjecture by Hegselmann and
    Krause in opinion dynamics, concerning a symmetric bounded
    confidence interval model: If there is a truth and all individuals
    take each other seriously by a positive amount bounded away from
    zero, then all truth seekers will converge to the truth. Here
    truth seekers are the individuals which are attracted by the truth
    by a positive amount. In the absence of truth seekers it was
    already shown by Hegselmann and Krause that the opinions of the
    individuals converge.
  \end{abstract}
  \noindent
  \rule{\textwidth}{0.3 mm}

\section{Introduction}

We answer in the affirmative a conjecture posed by Hegselmann and
Krause about the long-term behavior of opinions in a finite group of
individuals, some of them attracted to the truth, the so-called
\emph{truth seekers}.  Our contribution: Under mild assumptions, the
opinions of all truth seekers converge to the truth, despite being
distracted by individuals not attracted to the truth, the
\emph{ignorants}.

The underlying model for opinion dynamics is the
\emph{bounded-confidence model}: Opinions, which themselves are
represented by real numbers in the unit interval, are influenced by
the opinions of others by means of averaging, but only if not too far
away. This bounded-confidence model (formal definitions below) was
first suggested by Krause in 1997.  It received a considerable amount
of attention in the artificial societies and social simulation
community \cite{lorenz_survey,hegselmann_new,
  hegselmann2006,hegselmann2002,15.1021F,7291M}.

The concept of truth seekers was invented in 2006 by Hegselmann and
Krause~\cite{hegselmann2006}, along with a philosophical discussion
about the scientific context with respect to the notion of truth.  We
blind out the philosophical discussions here and focus on the
resulting dynamical system, governed by difference equations that we
find interesting in their own right.

The opinions of \emph{truth seekers} are not only attracted by
opinions of others; they are additionally attracted by a constant
number, the truth.  The resulting opinion is weighted average of the
result of the original bounded-confidence dynamics and the truth.
Individuals not attracted by the truth in this sense are
\emph{ignorants}. In their paper, Hegselmann and Krause show that if
all individuals are truth seekers -- no matter how small the weight
--, then (the opinions of) all the individuals converge to consensus
on the truth value.

The question we answer in this paper arises when some of the
individuals are ignorants, i.e.{}, the weight of the influence of the
truth is zero for them.  Numerous simulation experiments led
Hegselmann and Krause to the conjecture, that still the opinions of
all the truth seekers finally end up at the truth.  However, a proof
of this fact could not be found so far.  Evidence by simulation only,
however, bears the risk of numerical artefacts -- very much so in the
non-continuous bounded-confidence model.  Therefore, it is desirable
to provide mathematically rigid proofs of structural properties of
bounded-confidence dynamics.

Allthough the conjecture may seem self-understood at first glance
because of the contraction property of the system dynamics for truth
seekers, a second look on the situation reveals that the conjecture
and its confirmation in this paper are far from trivial: several
innocent-looking generalizations of the conjecture are actually false,
as we will show below in the technical parts of the paper.  Relying on
intuition only is dangerous.

Even in the affirmative cases, convergence turns out to be quite slow
in general and far from monotone.  The main difficulty is the
following: the convergence of truth seekers heavily depends on their
long-term influence on ignorants.  Depending on the configuration of
ignorants and the parameters of the system, there are arbitrarily many
iterations in which the truth seekers deviate from the truth.  The
crucial observation is that, during these iterations, the
configuration of ignorants is somehow ``improved'' because the truth
seekers attract them.

After all, the proof is elementary but extremely technical.  We
introduce some structures like the \emph{confidence graph}, that might
prove useful also in other contexts.  Other structures we need are
rather special, probably with limited use beyond this paper.  It
would, therefore, be desirable to find a more elegant proof, revealing
the reason why the conjecture is true.  For example: find a suitable
Lyapunov function.  The examples we give as we go along in the proof,
however, indicate that a certain amount of complexity has to be
captured by the arguments because the line between true and false
conjectures is extremely thin.

\section{Formal problem statement}
\label{sec:form-probl-stat}

Suppose there is a set $[n]:=\{1,\dots,n\}$ of individuals with
opinions $x_i(t)\in[0,1]$ at time $t$ for all $i\in[n]$,
$t\in\mathbb{N}$. The abstract truth is modeled as a constant over
time, denoted by $h\in[0,1]$. The opinion of an individual $i\in[n]$ is
influenced in a time step $t$ only by those individuals which have a
similar opinion, more precisely which have an opinion in the
\textit{confidence interval} of $x_i(t)$.

\begin{definition}
  For $x\in [0,1]$ and a parameter $\varepsilon\ge 0$ we define the
  \emph{confidence set of value} $x$ \emph{at time} $t$ as
  $$
    I_x^\varepsilon(t):=\{j\in[n]\mid |x-x_j(t)|\le \varepsilon\}.
  $$
  As a shorthand we define
  $I_i^\varepsilon(t):=I_{x_i(t)}^\varepsilon(t)$ for any $i\in[n]$.
\end{definition}
The update of the opinions is modeled as a weighted arithmetic mean of
opinions in the confidence set and a possible attraction towards the
truth.
\begin{definition}
  A \emph{weighted arithmetic mean symmetric bounded
    confidence opinion system} (WASBOCOS) is a tupel 
  \begin{equation*}
    (n, h, \varepsilon, \alpha, \beta; \alpha_i(t), \beta_{ij}(t), x_i(0)),
  \end{equation*}
  where
  \begin{itemize}
  \item $n\in\mathbb{N}$ ist the number of \emph{individuals},
  \item $h\in[0,1]$ ist the \emph{truth},
  \item $\varepsilon\in[0,1]$ is the \emph{bounded confidence radius},
  \item $\alpha\in(0,1]$ is a lower bound for the weight of the truth for truth seekers,
  \item $\beta\in(0,\frac{1}{2}]$ is a lower bound for the weight of opinions in the bounded confidence interval,
  \item $\alpha_i(t)\in[\alpha,1]$ or $\alpha_i(t)=0$ for all
    $t\in\mathbb{N}$ is the actual weight of the truth for truth
    seeker~$i$ at time step~$t$,
  \item $\beta_{ij}(t)\in[\beta,1-\beta]$ with $\sum_{i=j}^n
    \beta_{ij} = 1$ for all $i\in[n]$ and for all $t\in\mathbb{N}$ is
    the weight of opinion~$j$ in the view of agent~$i$,
  \item $x_i(0)\in[0,1]$ is the starting opinion of Individual~$i$.
  \end{itemize}
  The \emph{bounded confidence dynamics} on such a system is defined
  by simultaneous updates of the opinions in the following form:
  \begin{equation}
    \label{eq_update}
    x_i(t+1):=\alpha_i(t)\cdot h+\Big(1-\alpha_i(t)\Big)\frac{\sum\limits_{j\in I_i^\varepsilon(t)}\beta_{ij}(t)x_j(t)}
    {\sum\limits_{j\in I_i^\varepsilon(t)}\beta_{ij}(t)}.
  \end{equation}
  \emph{Individuals} are members of the index set $[n]$.
  \emph{Truth seekers} are members of the set $K:=\{k\in[n]\mid
  \alpha_k(t)\ge \alpha\, \forall t\in\mathbb{N}\}$.  All other
  individuals, i.e.{}, those with $\alpha_i(t)=0$ for all $t$, are
  called \emph{ignorants}; their set is denoted by~$\overline{K}$.
\end{definition}

See Figure~\ref{fig:example-0} for a sketch of a typical set of
trajectories.  Remark: The term \emph{symmetric} in the notion of a
(WASBOCOS) refers to the confidence radius, not to the weights that
individuals assign to other individuals' opinions.
\begin{figure}
  \begin{center}
  \includegraphics[width=0.65\linewidth]{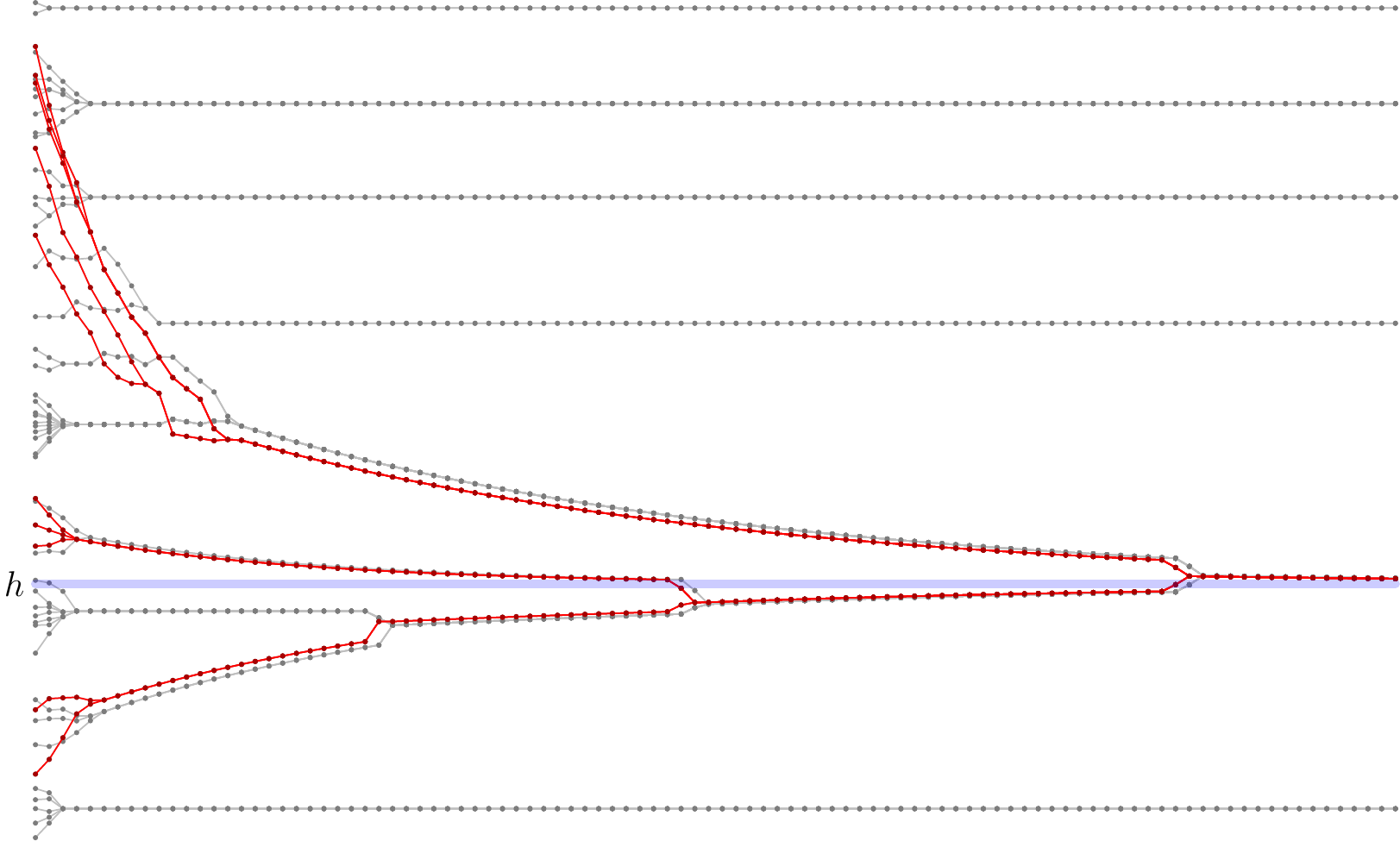}
  \label{fig:example-0}
  \end{center}
\end{figure}
The main result we wish to prove is the following:
\begin{theorem}(Generalized Hegselmann-Krause Conjecture) 
  \label{main_result}
  All truth seekers in an (WASBOCOS) $\Omega$ converge to the
  truth~$h$. Formally, for each $\gamma>0$ and each $\Omega$ there
  exists a $T(\gamma,\Omega)$ so that we have $|x_k(t)-h|<\gamma$ for
  all $k\in K$ and all $t\ge T(\gamma,\Omega)$.
\end{theorem}
Note that we use $\gamma>0$ in the statement of convergence instead
of~$\varepsilon>0$ because $\varepsilon$ is traditionally used for the
bounded confidence radius.

It is important that convergence is not just implied by the
contraction property of the dynamics with ignorants ignored.
Ignorants and where their opinions are make a huge difference (see
Figure~\ref{fig:example-1} for an example).
\begin{figure}
  \centering
  \includegraphics[width=0.65\linewidth]{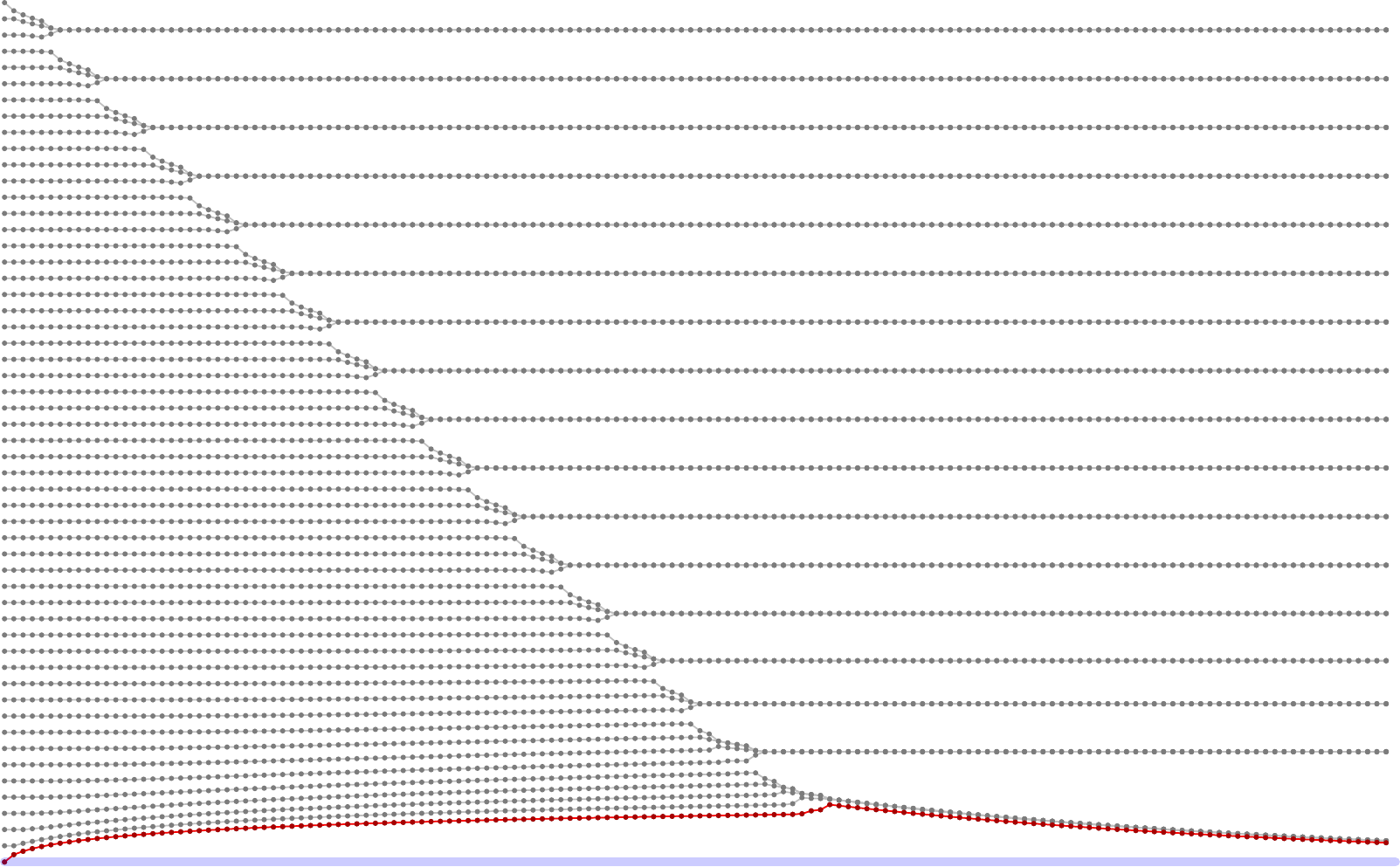}
  \caption{The only truth seeker is located at the bottom on the
    truth; it gets attracted away from the truth for quite some time;
    eventually, the ignorants either are ``converted'' or left
    behind. Thus, we cannot expect monotone convergence of the truth
    seeker farthest from the truth, i.p, reaching the truth once does
    not imply convergence to the truth.}
  \label{fig:example-1}
\end{figure}

It would be nice if one could derive a bound on the speed of
convergence, i.e., a bound on $T(\gamma,\Omega)$, in terms of the
structural parameters $\varepsilon$, $\alpha$, $\beta$, and $n$.
Unfortunately, this is not possible.  The speed of convergence is not
determined by the structural parameters alone.  This can be seen in
the following simple example.
\begin{example}
  \label{ex_interupted_convergent}
  Consider a (WASBOCOS) with truth $h=\varepsilon$, $\varepsilon>0$,
  $\alpha_1(t)=\alpha$, $\alpha_2(t)=0$, $\beta_{ij}(t)=\frac{1}{2}$,
  $\beta=\frac{1}{2}$, $x_1(0)=2\varepsilon$,
  $x_2(0)=\tilde{\varepsilon}$, where
  $\varepsilon>\tilde{\varepsilon}>0$. Let $T\in\mathbb{N}$ be the
  smallest integer so that
  $(1-\alpha)^T\varepsilon\le\tilde{\varepsilon}$. Then by induction
  we have $x_1(t)=\varepsilon+(1-\alpha)^t\varepsilon$ and
  $x_2(t)=x_2(0)=\tilde{\varepsilon}$ for all $t\le T$. So truth
  seeker $1$ seems to monotonically converge to the truth, but at time
  $T+1$ we have
  $x_1(T+1)=\alpha\varepsilon+\frac{1-\alpha}{2}\left(\varepsilon+(1-\alpha)^T\varepsilon+\tilde{\varepsilon}\right)
  \le \frac{1+\alpha}{2}\cdot\varepsilon+\tilde{\varepsilon}$.
\end{example}
See Figure~\ref{fig:example-2} for a sketch of the situation.
\begin{figure}
  \centering
  \includegraphics[width=0.65\linewidth]{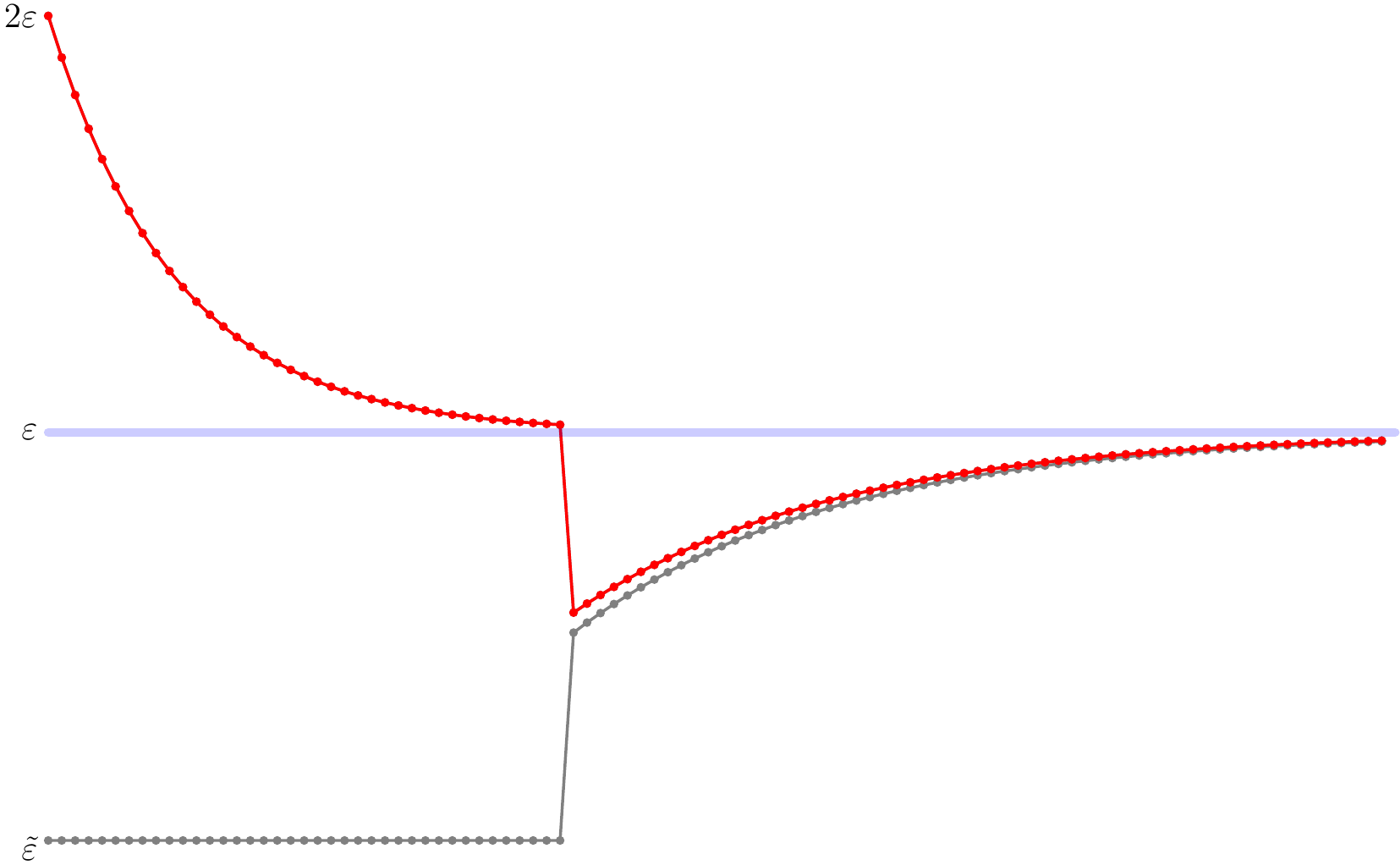}
  \caption{A sketch of interrupted convergence: a lonely truth seeker
    starting at $2\varepsilon$ seems to monotonically converge to the
    truth right away, but suddenly its confidence interval picks up
    the ignorant on the other side of the truth, and the truth seeker
    gets distracted.  However, finally the ignorant gets distracted
    himself, and convergence is eventually established.}
  \label{fig:example-2}
\end{figure}
Since we may choose $\tilde{\varepsilon}$ arbitrarily small, we find
the following: in general, we can not expect that
for every $\gamma>0$ there is a $T(\gamma,\varepsilon,\alpha,\beta,n)$
such that for all $t\ge T(\gamma,\varepsilon,\alpha,\beta,n)$ we have
$|x_k(t)-h|<\gamma$ for each truth seeker $k\in K$.  But we may have
\emph{interrupted convergence}: In a first phase, the truth seekers
come arbitrarily close to the truth in time only dependent on the
structural parameters; then, they may temporarily get distracted at
some point; finally, they converge to the truth in time depending only
on the structural parameters \emph{and the time of distraction}.  This
can be formalized as follows:
\begin{definition}
  \label{def_interupted_convergent}
  Given $\varepsilon$, $\alpha$, $\beta$, $n$, we say that truth
  seekers $k\in K$ are (1-fold) \emph{interrupted convergent to the
    truth}, if for each $\gamma > 0$ there exist two functions
  $T_1^s(\gamma,\varepsilon,\alpha,\beta,n)$ and
  $T_2^s(\gamma,\varepsilon,\alpha,\beta,n,T_1^e)$, so that for each
  (WASBOCOS) $\Omega$, with structural parameters $\varepsilon$,
  $\alpha$, $\beta$ and $n$, there exists an $T_1^e\in\mathbb{N}$
  satisfying
  \begin{eqnarray*}
    &&\forall k\in K,\,\forall t\in[T_1^s(\gamma,\varepsilon,\alpha,\beta,n),T_1^e]:\, |x_k(t)-h|<\gamma,\\
    &&\forall k\in K,\,\forall t\ge T_2^s(\gamma,\varepsilon,\alpha,\beta,n,T_1^e):\,|x_k(t)-h|<\gamma.
  \end{eqnarray*}
\end{definition}

Theorem \ref{main_result} is now a corollary of the following
substantially strengthened Theorem:
\begin{theorem}
  \label{thm_interupted_convergent}
  All truth seekers in an (WASBOCOS) $\Omega$ are (1-fold) interrupted
  convergent to the truth.
\end{theorem}

Originally Hegselmann and Krause considered the (WASBOCOS) model for
$\alpha_{i}(t)\in\{0,\alpha\}$ and $\beta_{ij}(t)=\frac{1}{n}$.  In
the case of complete absence of truth seekers they have already
proved, that the opinion of each individual converges, as can be
expected, not necessarily to the truth. In fact in general the
individuals form several clusters, where two individuals of different
clusters converge to different opinions.

We give an example without truth seekers where the individuals will
converge to five different clusters.

\begin{example}
  \label{ex_1}
  Consider a (WASBOCOS) with $\alpha_i(t)=0$ (no truth seekers),
  $\beta=\beta_{ij}(t)=\frac{1}{n}$, $n=12$; the values of~$\alpha$
  and~$h$ do not matter.  The starting positions are given by
  \begin{eqnarray*}
   &x_1(0)=x_2(0)=0,\,
   x_3(0)=\varepsilon,\,
   x_4(0)=2\varepsilon,\,
   x_5(0)=3\varepsilon,\,
   x_6(0)=x_7(0)=4\varepsilon,\,\\
   &x_8(0)=5\varepsilon,\,
   x_9(0)=6\varepsilon,\,
   x_{10}(0)=7\varepsilon,\,
   \text{and}\,
   x_{11}(0)=x_{12}(0)=8\varepsilon,
  \end{eqnarray*}
  see Figure \ref{fig_ex_1}.

  \medskip

  In Table \ref{table_ex_1} we give the complete dynamics of the
  opinions of all $12$ individuals over time until the opinion of
  every individual has converged. For brevity we write $x_i$ instead
  of $x_i(t)$. After three time steps, see Figure~\ref{fig:example-5}
  for the dynamics, we have reached a stable state, see Figure
  \ref{fig_ex_1_2} for the resulting positions of the individuals.
\end{example}

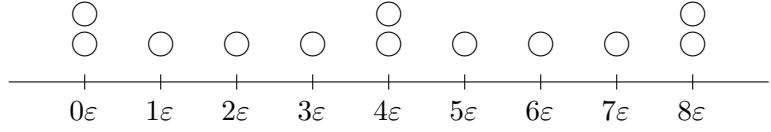
\begin{figure}[htp]
  \begin{center}
    \setlength{\unitlength}{1.0cm}
    \begin{picture}(8,1.4)
      \put(0,1){\circle{0.3}}
      \put(0,1.4){\circle{0.3}}
      \put(1,1){\circle{0.3}}
      \put(2,1){\circle{0.3}}
      \put(3,1){\circle{0.3}}
      \put(4,1){\circle{0.3}}
      \put(4,1.4){\circle{0.3}}
      \put(5,1){\circle{0.3}}
      \put(6,1){\circle{0.3}}
      \put(7,1){\circle{0.3}}
      \put(8,1){\circle{0.3}}
      \put(8,1.4){\circle{0.3}}
      \put(-1,0.5){\line(1,0){10}}
      \multiput(0,0)(1,0){9}{\put(0,0.4){\line(0,1){0.2}}}
      \put(-0.2,0){$0\varepsilon$}
      \put(0.8,0){$1\varepsilon$}
      \put(1.8,0){$2\varepsilon$}
      \put(2.8,0){$3\varepsilon$}
      \put(3.8,0){$4\varepsilon$}
      \put(4.8,0){$5\varepsilon$}
      \put(5.8,0){$6\varepsilon$}
      \put(6.8,0){$7\varepsilon$}
      \put(7.8,0){$8\varepsilon$}
    \end{picture}
    \caption{Starting positions of the individuals in Example \ref{ex_1}.}
    \label{fig_ex_1}
  \end{center}
\end{figure}

\begin{table}[htp]
  \begin{center}\renewcommand{\arraystretch}{1.5}
    \begin{tabular}{r@{\qquad}rrrrrrrrr}
      \toprule
      $t$ & $x_1=x_2$ & $x_3$ & $x_4$ & $x_5$ & $x_6=x_7$ & $x_8$ & $x_9$ & $x_{10}$ & $x_{11}=x_{12}$ \\
      \midrule
      0 & $0\varepsilon$ & $1\varepsilon$ & $2\varepsilon$ & $3\varepsilon$ & $4\varepsilon$ & $5\varepsilon$ &
      $6\varepsilon$ & $7\varepsilon$ & $8\varepsilon$ \\[0.9mm]
      1 & $\frac{1}{3}\varepsilon$ & $\frac{3}{4}\varepsilon$ & $2\varepsilon$ & $\frac{13}{4}\varepsilon$ & $4\varepsilon$
      & $\frac{19}{4}\varepsilon$ & $6\varepsilon$ & $\frac{29}{4}\varepsilon$ & $\frac{23}{3}\varepsilon$ \\[1.6mm]
      2 & $\frac{17}{36}\varepsilon$ & $\frac{17}{36}\varepsilon$ & $2\varepsilon$ & $\frac{15}{4}\varepsilon$ & $4\varepsilon$
      & $\frac{17}{4}\varepsilon$ & $6\varepsilon$ & $\frac{271}{36}\varepsilon$ & $\frac{271}{36}\varepsilon$ \\[1.6mm]
      3 & $\frac{17}{36}\varepsilon$ & $\frac{17}{36}\varepsilon$ & $2\varepsilon$ & $4\varepsilon$ & $4\varepsilon$ &
      $4\varepsilon$ & $6\varepsilon$ & $\frac{271}{36}\varepsilon$ & $\frac{271}{36}\varepsilon$\\
      \bottomrule
    \end{tabular}
    \medskip
    \caption{The dynamics of Example \ref{ex_1} in numbers.}
    \label{table_ex_1}
  \end{center}
\end{table}

\begin{figure}[htbp]
  \centering
  \includegraphics[width=0.65\linewidth]{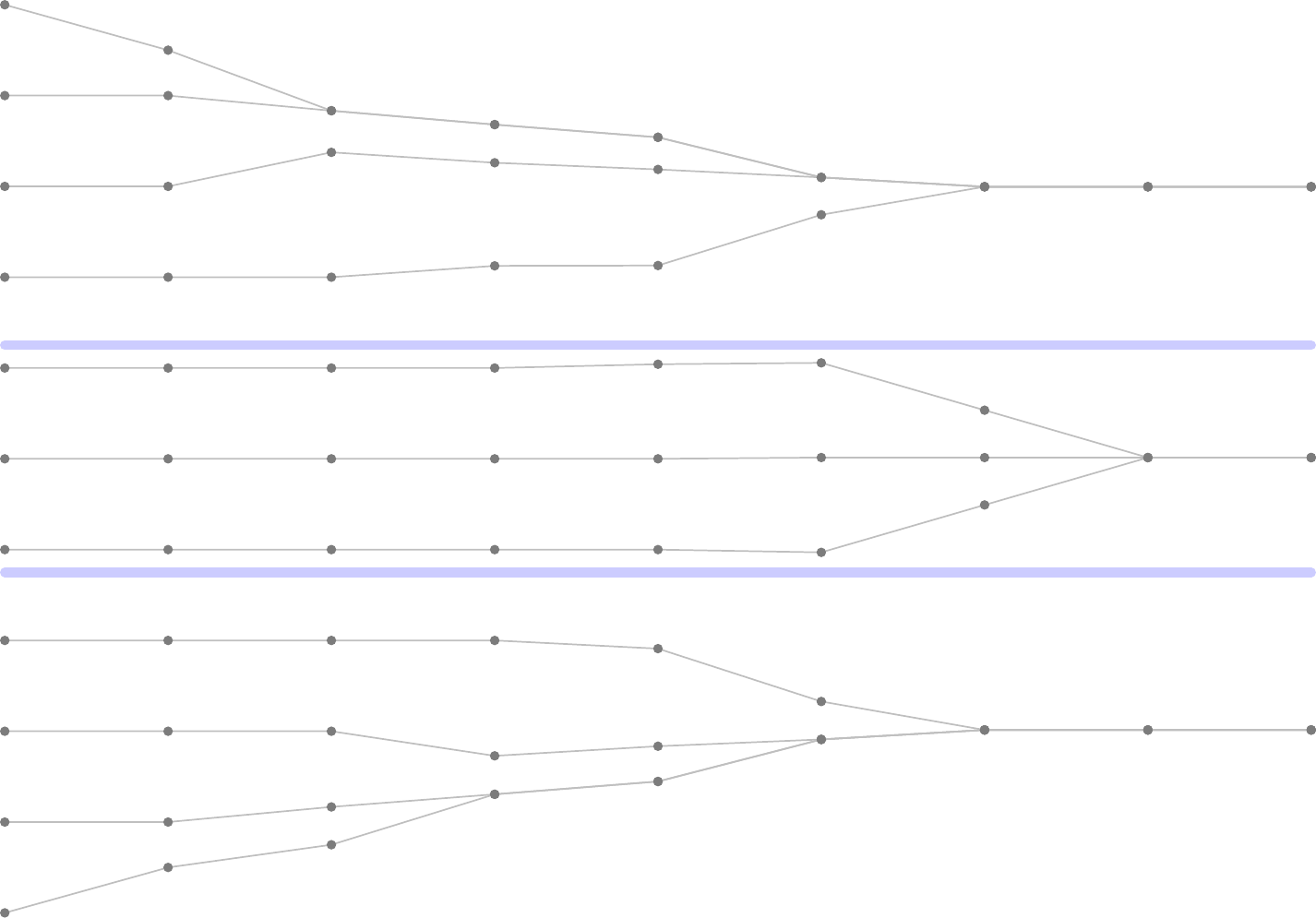}
  \caption{The dynamics in Example \ref{ex_1}.}
  \label{fig:example-5}
\end{figure}

\begin{figure}[!ht]
  \begin{center}
    \setlength{\unitlength}{1.0cm}
    \begin{picture}(8,2.6)
      \put(0.4722,1){\circle{0.3}}
      \put(0.4722,1.4){\circle{0.3}}
      \put(0.4722,1.8){\circle{0.3}}
      \put(2,1){\circle{0.3}}
      \put(4,1){\circle{0.3}}
      \put(4,1.4){\circle{0.3}}
      \put(4,1.8){\circle{0.3}}
      \put(4,2.2){\circle{0.3}}
      \put(6,1){\circle{0.3}}
      \put(7.5277,1){\circle{0.3}}
      \put(7.5277,1.4){\circle{0.3}}
      \put(7.5277,1.8){\circle{0.3}}
      \put(-1,0.5){\line(1,0){10}}
      \multiput(0,0)(1,0){9}{\put(0,0.4){\line(0,1){0.2}}}
      \put(-0.2,0){$0\varepsilon$}
      \put(0.8,0){$1\varepsilon$}
      \put(1.8,0){$2\varepsilon$}
      \put(2.8,0){$3\varepsilon$}
      \put(3.8,0){$4\varepsilon$}
      \put(4.8,0){$5\varepsilon$}
      \put(5.8,0){$6\varepsilon$}
      \put(6.8,0){$7\varepsilon$}
      \put(7.8,0){$8\varepsilon$}
    \end{picture}
    \caption{Final positions of the individuals in Example \ref{ex_1}.}
    \label{fig_ex_1_2}
  \end{center}
\end{figure}
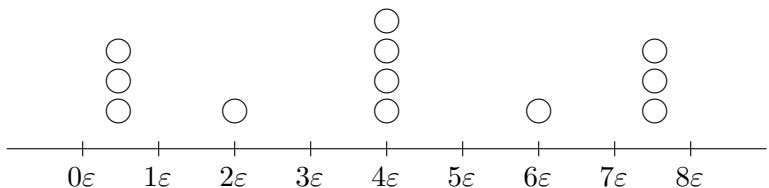

We remark that for symmetric weights $\beta_{ij}(t)=\beta_{ji}(t)$ one
can easily show that in the absence of truth seekers the dynamics
becomes stable after a finite number of time steps. In the case of
asymmetric weights $\beta_{ij}(t)\neq\beta_{ji}(t)$ we only have
convergence, but need not reach a stable state after an arbitrary,
problem dependent, but finite number of time steps, as illustrated in
the following example.

\begin{example}
  \label{ex_2}
  Consider a (WASBOCOS) with $\alpha_i(t)=0$ (no truth seekers),
  $n=2$, $x_1(0)=0$, $x_2(0)=\varepsilon$,
  $\beta_{11}(t)=\frac{2}{3}$, $\beta_{12}(t)=\frac{1}{3} = \beta$,
  $\beta_{21}(t)=\frac{1}{2}$, $\beta_{22}(t)=\frac{1}{2}$; the values
  of $\alpha$ and~$h$ do not matter.

  \medskip

  One can easily verify, e.g.{}, by induction, that we have
  $$
    x_1(t)=\left(\frac{2}{5}-\frac{2}{5\cdot 6^t}\right)\cdot\varepsilon\text{ and }
    x_2(t)=\left(\frac{2}{5}+\frac{3}{5\cdot 6^t}\right)\cdot\varepsilon
  $$
  for all $t\in\mathbb{N}$. So we have $|x_1(t)-x_2(t)|=\frac{1}{6^t}\cdot\varepsilon>0$ but
  clearly the opinions of the two individuals converge to $\frac{2}{5}$.
\end{example}
All stated insights with the absence of truth seekers were known so
far. It becomes a bit more interesting if we allow truth seekers,
i.e.{}, if we consider a general (WASBOCOS).
\begin{example}
  \label{ex_3}
  Consider a (WASBOCOS) with $\alpha_1(t)=\alpha$, $\alpha_i(t)=0$ for $i\neq 1$, $\beta_{ij}(t)=\frac{1}{n}$,
  $h=\frac{1}{2}\varepsilon$, $x_1(0)=0$, and $x_i(0)=\varepsilon$ for $i\neq 1$. The opinion $u_t$ of the
  truth seeker $1$ at time 
  $t$ and the opinion $v_t$ of the other ignorants at time $t>0$ are given by
  \begin{eqnarray*}
    u_t&=&\left[\frac{1}{2}-\alpha
    \left(\frac{1}{2}-\frac{1}{n}\right)\left(1-\frac{\alpha}{n}\right)^{t-1}\right]\varepsilon,\\
    v_t&=&\left[\frac{1}{2}+
    \left(\frac{1}{2}-\frac{1}{n}\right)\left(1-\frac{\alpha}{n}\right)^{t-1}\right]\varepsilon
  \end{eqnarray*}
  respectively. This can be verified, e.g.{}, by induction. We see
  that the opinions of the truth seekers, and here also those of the
  ignorants, converge to the truth $h=\frac{1}{2}\varepsilon$.

  Note the opinions of ignorants may in general fail to converge to
  the truth as one can see by adding some further ignorants with
  $\tilde{x}_i(0)=3\varepsilon$.
\end{example}
As our analytical investigation of the previous example was rather
technical, we also depict the situation for special values $n=6$ and
$\alpha=\frac{2}{3}$ in Figure \ref{fig_ex_3}.  We sketch the truth
seeker by a filled circle and the ignorants by an empty circle.

\begin{figure}[h]
  \begin{center}
    \includegraphics[width=0.65\linewidth]{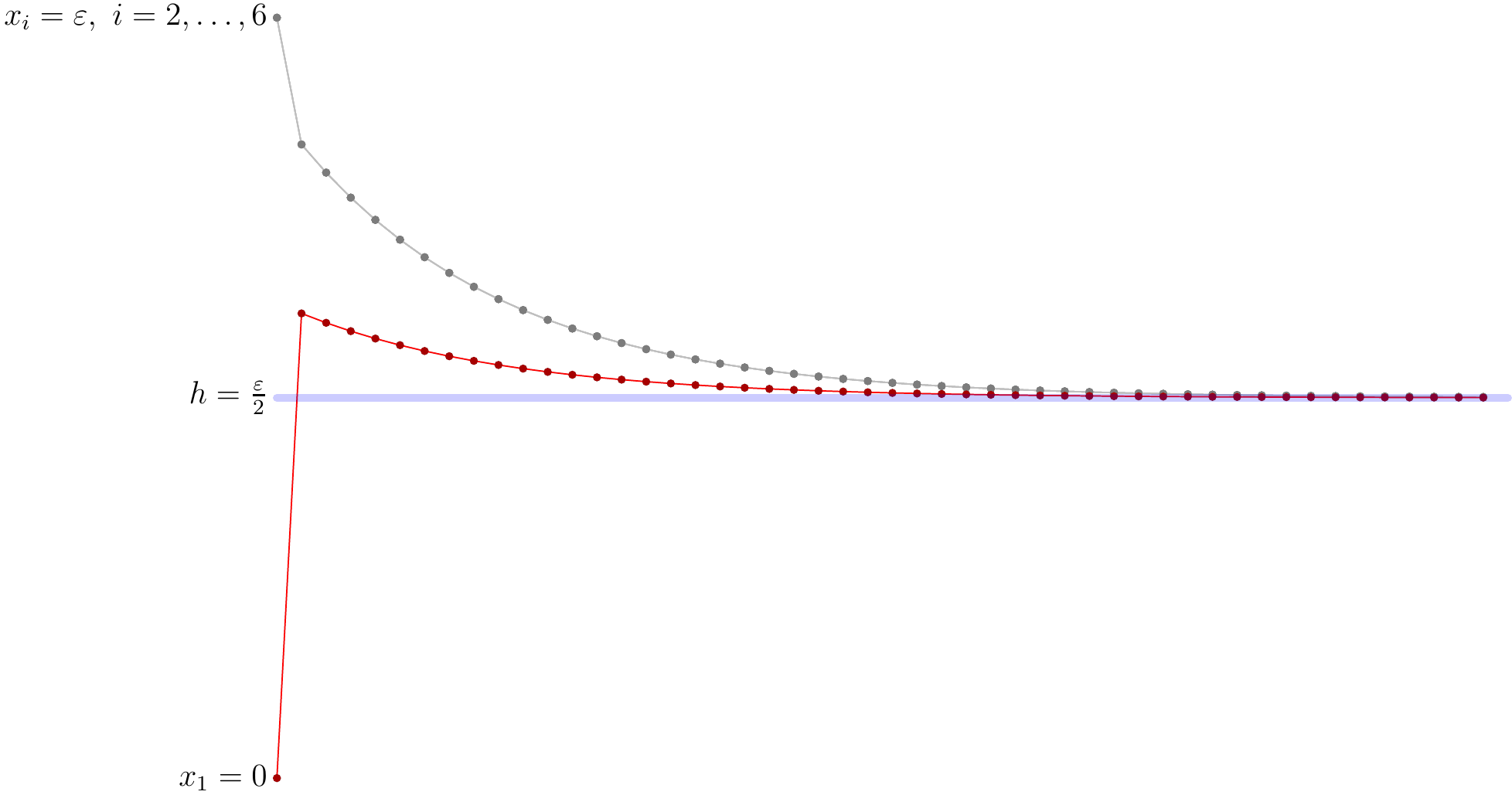}
    \caption{The dynamics in example \ref{ex_3}.}
    \label{fig_ex_3}
  \end{center}
\end{figure}

One can easily imagine more complicated configurations as in Example
\ref{ex_3} where one has little chance and willingness to describe the
situation analytically. Our result Theorem \ref{main_result} states
that -- whatever the parameters of a (WASBOCOS) are -- the opinions of the
truth seekers converge to the truth. This settles an open conjecture
of Hegselmann and Krause.


\section{The crucial objects}

To get a first impression of what we may expect in terms of
convergence we consider a lonely truth seeker, i.e.{}, $n=1$.

\begin{lemma}
  \label{lemma_lonely}
  For a lonely truth seeker $i=1$ we have 
  $$
    |x_i(t+r)-h|\le |x_i(t)-h|\cdot(1-\alpha)^r.
  $$
\end{lemma}
\begin{proof}
  $$
    |x_i(t+1)-h|=|x_i(t)-h|\cdot(1-\alpha_i(t))\le |x_i(t)-h|\cdot(1-\alpha).
  $$
\end{proof}

Clearly this bound is tight. Similar to this very special situation of
a lonely truth seeker is the case $\varepsilon=0$, so that we now
assume $\varepsilon>0$ for the remaining part of this article.

In order to describe the states of the discrete time dynamical system
with more than one truth seeker, we look at the truth seekers with the
most extreme opinions.

\begin{definition}
  \label{def_lower_upper}
  We define $\tilde{u}(t)\in K$ as the lexicographically smallest
  truth seeker which fullfills $x_{\tilde{u}(t)}(t)\ge h$ and
  $x_{\tilde{u}(t)}(t)\ge x_k(t)$ for all $k\in K$. If there is no
  truth seeker with opinion greater or equal to the truth $h$ we set
  $\tilde{u}(t)=0$.  In order to avoid case distinctions, we define
  $x_0(t'):=h$ for all $t'\in\mathbb{N}$. Similar we define
  $\tilde{l}(t)$ as the lexicographically smallest truth seeker that
  fullfills $x_{\tilde{l}(t)}(t)\le h$ and $x_{\tilde{l}(t)}(t)\le
  x_k(t)$ for all $k\in K$. Again, we set $\tilde{l}(t)=0$ if
  there is no such truth seeker.
\end{definition}
Due to the \emph{symmetrical} -- one could say \textit{fair} --
definition of the confidence set, the confidence structure between the
individuals can be described as a simple graph with loops.
\begin{definition}
  The \emph{confidence graph} $\mathcal{G}(t)$ with vertex set $V(t)$
  and edge set $E(t)$, of a configuration $x(t)=(x_1(t) , \dots ,
  x_n(t))\in\mathbb{R}^n$ and the additional is defined as follows:
  \begin{eqnarray*}
    V(t):=[n]\cup\{0\},\\
    E(t):=\{\{i,j\} \in \tbinom{V}{2} \mid |x_i(t)-x_j(t)|\le\varepsilon\}.
  \end{eqnarray*}
  For $i\in V(t)$ let $C_i(t)$ be the set of vertices in the
  connectivity component of vertex $i$ in $\mathcal{G}(t)$.
\end{definition}

Because we want to keep track of the individuals which can influence
the truth seekers in the future, we give a further definition for
individuals, which is similar to Definition \ref{def_lower_upper} for
truth seekers.

\begin{definition}
  \label{def_extreme_individuals}
  We define $\hat{u}(t)\in C_{\tilde{u}(t)}(t)$ as the
  lexicographically smallest individual with $x_{\hat{u}(t)}(t)\ge
  x_c(t)\,\,\forall c\in C_{\tilde{u}(t)}(t)$ and $\hat{l}(t)\in
  C_{\tilde{l}(t)}(t)$ as the lexicographically smallest individual
  with $x_{\hat{l}(t)}(t)\le x_c(t)\,\,\forall c\in
  C_{\tilde{l}(t)}(t)$ for all $t\in\mathbb{N}$.
\end{definition}

The opinions of $\hat{u}(t)$ and $\hat{l}(t)$ form an interval
$[x_{\hat{l}(t)}(t),x_{\hat{u}(t)}(t)]$ called the \emph{hope
  interval} which is crucial for our further investigations. To prove
the main theorem we will show that the length of this hope interval
converges to zero.

\medskip

In Figure \ref{fig_def_borders}, we have depicted a configuration to
illustrate Definition \ref{def_lower_upper} and Definition
\ref{def_extreme_individuals}. In particular, we have $\tilde{l}=4$,
$\tilde{u}=9$, $\hat{l}=2$, and $\hat{u}=12$. Individual~$1$ is
\textit{lost} and not contained in the hope interval, because there is
no path in $\mathcal{G}$ from $1$ to $\tilde{l}=4$. So we already know
that the opinion of Individual~$1$ will not converge to the truth.

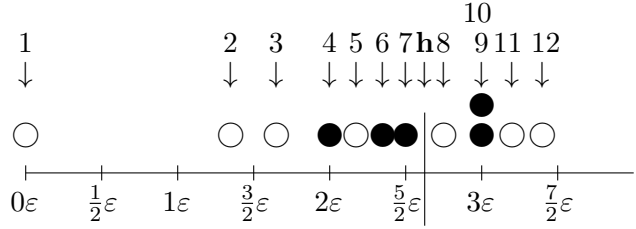
\begin{figure}[htp]
  \begin{center}
    \setlength{\unitlength}{1.0cm}
    \begin{picture}(6,3.0)
      \put(-1,1){\circle{0.3}}
      \put(1.7,1){\circle{0.3}}
      \put(2.3,1){\circle{0.3}}
      \put(3,1){\circle*{0.3}}
      \put(3.35,1){\circle{0.3}}
      \put(3.7,1){\circle*{0.3}}
      \put(4,1){\circle*{0.3}}
      \put(4.5,1){\circle{0.3}}
      \put(5,1){\circle*{0.3}}
      \put(5,1.4){\circle*{0.3}}
      \put(5.4,1){\circle{0.3}}
      \put(5.8,1){\circle{0.3}}
      \put(-1.1,1.7){$\downarrow$}
      \put(-1.1,2.1){$1$}
      \put(1.6,1.7){$\downarrow$}
      \put(1.6,2.1){$2$}
      \put(2.2,1.7){$\downarrow$}
      \put(2.2,2.1){$3$}
      \put(2.9,1.7){$\downarrow$}
      \put(2.9,2.1){$4$}
      \put(3.25,1.7){$\downarrow$}
      \put(3.25,2.1){$5$}
      \put(3.6,1.7){$\downarrow$}
      \put(3.6,2.1){$6$}
      \put(3.9,1.7){$\downarrow$}
      \put(3.9,2.1){$7$}
      \put(4.4,1.7){$\downarrow$}
      \put(4.4,2.1){$8$}
      \put(4.9,1.7){$\downarrow$}
      \put(4.9,2.1){$9$}
      \put(4.75,2.5){$10$}
      \put(5.3,1.7){$\downarrow$}
      \put(5.15,2.1){$11$}
      \put(5.7,1.7){$\downarrow$}
      \put(5.65,2.1){$12$}
      \put(-1,0.5){\line(1,0){8}}
      \multiput(0,0)(1,0){8}{\put(-1,0.4){\line(0,1){0.2}}}
      \put(-1.2,0){$0\varepsilon$}
      \put(-0.2,0){$\frac{1}{2}\varepsilon$}
      \put(0.8,0){$1\varepsilon$}
      \put(1.8,0){$\frac{3}{2}\varepsilon$}
      \put(4.25,-0.2){\line(0,1){1.5}}
      \put(4.15,1.7){$\downarrow$}
      \put(4.13,2.1){$\mathbf{h}$}
      \put(2.8,0){$2\varepsilon$}
      \put(3.8,0){$\frac{5}{2}\varepsilon$}
      \put(4.8,0){$3\varepsilon$}
      \put(5.8,0){$\frac{7}{2}\varepsilon$}
    \end{picture}
    \caption{Illustration of Definition \ref{def_lower_upper} and Definition \ref{def_extreme_individuals}.}
    \label{fig_def_borders}
  \end{center}
\end{figure}

In the configuration depicted in Figure \ref{fig_def_borders_2} we have $\tilde{l}=2$,
$\tilde{u}=0$, $\hat{l}=2$, and $\hat{u}=5$.

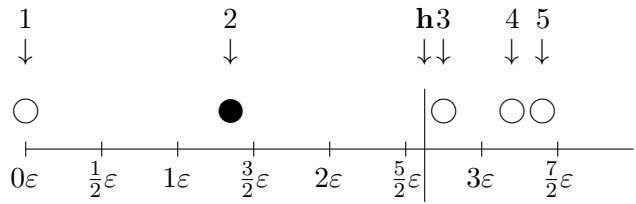
\begin{figure}[htp]
  \begin{center}
    \setlength{\unitlength}{1.0cm}
    \begin{picture}(6,3.0)
      \put(-1,1){\circle{0.3}}
      \put(1.7,1){\circle*{0.3}}
      \put(4.5,1){\circle{0.3}}
      \put(5.4,1){\circle{0.3}}
      \put(5.8,1){\circle{0.3}}
      \put(-1.1,1.7){$\downarrow$}
      \put(-1.1,2.1){$1$}
      \put(1.6,1.7){$\downarrow$}
      \put(1.6,2.1){$2$}
      \put(4.4,1.7){$\downarrow$}
      \put(4.4,2.1){$3$}
      \put(5.3,1.7){$\downarrow$}
      \put(5.31,2.1){$4$}
      \put(5.7,1.7){$\downarrow$}
      \put(5.71,2.1){$5$}
      \put(-1,0.5){\line(1,0){8}}
      \multiput(0,0)(1,0){8}{\put(-1,0.4){\line(0,1){0.2}}}
      \put(-1.2,0){$0\varepsilon$}
      \put(-0.2,0){$\frac{1}{2}\varepsilon$}
      \put(0.8,0){$1\varepsilon$}
      \put(1.8,0){$\frac{3}{2}\varepsilon$}
      \put(4.25,-0.2){\line(0,1){1.5}}
      \put(4.15,1.7){$\downarrow$}
      \put(4.13,2.1){$\mathbf{h}$}
      \put(2.8,0){$2\varepsilon$}
      \put(3.8,0){$\frac{5}{2}\varepsilon$}
      \put(4.8,0){$3\varepsilon$}
      \put(5.8,0){$\frac{7}{2}\varepsilon$}
    \end{picture}
    \caption{Illustration of a special case in Definition \ref{def_lower_upper}.}
    \label{fig_def_borders_2}
  \end{center}
\end{figure}

Note that the weights $\beta_{ij}$ may be assymmetric.  Thus, the
sequence of the opinions of the individuals may reorder during the
time steps. As an example, consider, e.g.{}, three ignorants with
starting positions $x_1(0)=1\varepsilon$,
$x_2(0)=\frac{3}{2}\varepsilon$, and $x_3(0)=2\varepsilon$. The
weights may be given as $\beta_{11}(0)=0.01$, $\beta_{12}(0)=0.01$,
$\beta_{13}(0)=0.98$, $\beta_{21}(0)=0.98$, $\beta_{22}(0)=0.01$,
$\beta_{23}(0)=0.01$, $\beta_{31}(0)=0.4$, $\beta_{32}(0)=0.4$, and
$\beta_{33}(0)=0.2$. After one time step the new opinions are given by
$x_1(1)=1.985\varepsilon$, $x_2(1)=1.015\varepsilon$, and
$x_3(1)=1.4\varepsilon$. We remark that it is possible to achieve
every ordering of the three opinions in one time step by choosing
suitable weights $\beta_{ij}$ in this example. Nevertheless we have
the following straight-forward lemma:

\begin{lemma}
  \label{lemma_interval_1}
  Let $i$ be an ignorant, $l\in I_i^\varepsilon(t)$ be an individual
  with smallest opinion and $u\in I_i^\varepsilon(t)$ be an individual
  with largest opinion then we have $x_i(t+1)\in[x_l(t),x_u(t)]$.
\end{lemma}
\begin{proof}
  This follows directly from the system dynamics in Equation
  (\ref{eq_update}).
\end{proof}
For truth seekers we have a similar lemma:
\begin{lemma}
  \label{lemma_interval_2}
  Let $i$ be a truth seeker, $l\in I_i^\varepsilon(t)$ be an
  individual with smallest opinion, and let $u\in I_i^\varepsilon(t)$
  be an individual with largest opinion. For $x_i(t)\le h$ we have
  $x_i(t+1)\in[x_l(t),\max(h,x_u(t))]$ and for $x_i(t)\ge h$ we have
  $x_i(t+1)\in[\min(h,x_l(t)),x_u(t)]$.
\end{lemma}
Our goal is to prove that the length of the hope interval converges to
zero. To this end, we show first that the length does not increase
after an iteration of Equation (\ref{eq_update}).
\begin{lemma}
  \label{lemma_decreasing}
  For all time steps $t\in\mathbb{N}$ we have $x_{\hat{u}(t+1)}(t+1)\le x_{\hat{u}(t)}(t)$ and
  $x_{\hat{l}(t+1)}(t+1)\ge x_{\hat{l}(t)}(t)$.
\end{lemma}
\begin{proof}
  We only prove the last inequality since the proof is symmetric for
  the first inequality. Due to Definition
  \ref{def_extreme_individuals}, we have $x_{\hat{l}(t+1)}(t+1)\le h$
  and $x_{\hat{l}(t)}(t)\le h$. By $\mathcal{L}(t)$ we denote the set
  of individuals with opinion strictly smaller than
  $x_{\hat{l}(t)}(t)$.  That is, $\mathcal{L}(t'):=\{i\in[n]\mid
  x_i(t')<x_{\hat{l}(t')}(t')\}$ for all $t'\ge t$. We remark that, by
  definition, $\mathcal{L}(t')$ does not contain a truth seeker. We
  set $\mathcal{U}(t'):=[n]\backslash\mathcal{L}(t')$; this set
  contains the remaining individuals.

  \medskip

  Let $u$ be an individual in $\mathcal{L}(t)$ with the largest
  opinion. By applying Lemma \ref{lemma_interval_1} we get
  $x_i(t+1)\le x_u(t)$ for all $i\in\mathcal{L}(t)$. Now let $l$
  (e.g.{}, $l=\hat{l}(t)$) be an individual in $\mathcal{U}(t)$ with
  smallest opinion then by applying Lemma \ref{lemma_interval_1} and
  Lemma \ref{lemma_interval_2} we receive $x_i(t+1)\ge x_l(t)$ for all
  $i\in\mathcal{U}(t)$. Thus, we have $\hat{l}(t+1)\in\mathcal{U}(t)$
  and so $x_{\hat{l}(t+1)}(t+1)\ge x_{\hat{l}(t)}(t)$ follows.
\end{proof}

In the remaining part of this article we prove that the length of the
hope interval $|x_{\hat{u}(t)}(t)-x_{\hat{l}(t)}(t)|$ converges (in
some special sense) to zero, as $t$ tends to infinity.


\section{Proof of the Generalized Hegselmann-Krause Conjecture}

One difficulty in the proof arises from the fact that convergence
happens in two phases: in a first phase, the hope interval becomes
sufficiently small so that the confidence graph is the complete graph.
Then, it may happen that truth seekers approaching the truth from one
side get distracted to the other side of the truth.  At that point,
however, the confidence structure is so simple that all individuals in
the hope interval converge to the truth.  Since all truth seekers are
in the hope interval at all times, this proves the theorem.  Where
exactly we split the phases is a technical decisison.

First, we show that after a finite number $T_1$ of time steps,
depending only on $n$, $\varepsilon$, $\alpha$, and $\beta$, the hope
interval $[x_{\hat{l}(T_1)}(T_1),x_{\hat{u}(T_1)}(T_1)]$ is contained
in the interval
$[h-\varepsilon-\frac{\varepsilon\alpha\beta}{12},h+\varepsilon+\frac{\varepsilon\alpha\beta}{12}]$. Therefore,
we introduce the following notion.
\begin{definition}
  A \emph{good iteration} is an iteration where for $1\le r\le 3$
  one of the following conditions is fullfilled:
  \begin{enumerate}
    \item[(1)] the number of individuals in the hope interval decreases,
    \item[(2)] the opinion of $\hat{l}(t+r)$ reaches or passes $h-\varepsilon-\frac{\varepsilon\alpha\beta}{12}$,
    \item[(3)] the opinion of  $\hat{u}(t+r)$ reaches or passes $h+\varepsilon+\frac{\varepsilon\alpha\beta}{12}$,
    \item[(4)] $|x_{\hat{u}(t+r)}(t+r)-x_{\hat{u}(t)}(t)|\ge\frac{\varepsilon\alpha\beta^2}{12}$,
    \item[(5)] $|x_{\hat{l}(t+r)}(t+r)-x_{\hat{l}(t)}(t)|\ge\frac{\varepsilon\alpha\beta^2}{12}$.
  \end{enumerate}
\end{definition}

Clearly, there is only a finite number of good iterations. We may
choose $T_1=3\cdot\left(n+2\cdot
  1+2\cdot\frac{12}{\varepsilon\alpha\beta^2}\right)$. We formulate
the next two lemmas only for the lower bound $x_{\hat{l}(t)}(t)$
because analog arguments hold for $x_{\hat{u}(t)}(t)$. As a shorthand
we define $d(i,j,t):=\left|x_i(t)-x_j(t)\right|$. For each point in
time $t$ we define the sets
\begin{eqnarray*}
  &&\mathcal{N}(t):=\left\{i\in[n]\mid d(\hat{l}(t),i,t)\in\Big[0,\frac{\varepsilon\alpha\beta}{12}
  \Big)\right\},\\
  &&\mathcal{M}(t):=\left\{i\in[n]\mid d(\hat{l}(t),i,t)\in\Big[\frac{\varepsilon\alpha\beta}{12},
  \varepsilon\Big]\right\},
  \text{ and}\\
  &&\mathcal{F}(t):=\left\{i\in[n]\mid x_i(t)-x_{\hat{l}(t)}(t)>\varepsilon\right\}.
\end{eqnarray*}

\begin{lemma}
  \label{lemma_middle}
  If $\mathcal{M}(t)\neq\emptyset$ then there is a good iteration after $1$ step.
\end{lemma}
\begin{proof}
  We assume that there is an individual $j\in\mathcal{M}(t)$, i.e.{},
  $d(\hat{l}(t),j,t)\in\left[\frac{\varepsilon\alpha\beta}{12},\varepsilon\right]$.
  For the evaluation of Equation (\ref{eq_update}) for elements of
  $\mathcal{N}(t)$, $\mathcal{M}(t)$, or $\mathcal{F}(t)$ we do not
  need to consider the opinion of individuals in
  $[n]\backslash(\mathcal{N}(t)\cup\mathcal{M}(t)\cup\mathcal{F}(t))$. Let
  $i$ be an element of $\mathcal{N}(t)$ with opinion
  $x_i(t)=x_{\hat{l}(t)}+\delta$, where
  $0\le\delta<\frac{\varepsilon\alpha\beta}{12}$.  Let us first assume
  that $i$ is an ignorant. Due to Individual~$j$ we have
  \begin{eqnarray*}
    x_i(t+1)&\ge&
    x_i(t)-\underset{\text{individuals in }\mathcal{N}(t)\backslash\{i\}}
    {\underbrace{\delta\left(1-2\beta\right)}}
    +\underset{i}{\underbrace{0\cdot\beta}}
    +\underset{j}{\underbrace{
    \left(\frac{\varepsilon\alpha\beta}{12}-\delta\right)\cdot\beta}}\\
    &\ge& x_{\hat{l}(t)}+\frac{\varepsilon\alpha\beta^2}{12}.
  \end{eqnarray*}
  For a truth seeker we similarly get
  \begin{eqnarray*}
    x_i(t+1)&\ge&x_i(t)+\alpha\varepsilon+(1-\alpha)
    \left(-\delta\left(1-2\beta\right)+\left(\frac{\varepsilon\alpha\beta}{12}
    -\delta\right)\cdot\beta\right)\\
    &\ge&x_{\hat{l}(t)}+\frac{\varepsilon\alpha\beta^2}{12}.
  \end{eqnarray*}
  Now let $i$ be an element of $\mathcal{M}(t)\cup\mathcal{F}(t)$ with $x_i(t)=x_{\hat{l}(t)}+\delta$ where
  $\delta\ge\frac{\varepsilon\alpha\beta}{12}$. In any case ($i$ being a truth seeker or an ignorant) we have
  $$
    x_i(t+1)\ge x_{\hat{l}(t)}+\delta-\underset{\text{individuals with smaller opinion than
    $i$}}{\underbrace{\delta(1-\beta)}}+\beta\cdot 0\ge
    x_{\hat{l}(t)}+\frac{\varepsilon\alpha\beta^2}{12}.
  $$
\end{proof}

\begin{lemma}
  If $x_{\hat{l}(t)} <
  h-\varepsilon-\frac{\varepsilon\alpha\beta}{12}$ then after at least
  $3$ time steps we have a good iteration.
\end{lemma}
\begin{proof}
  Due to Lemma \ref{lemma_middle} we can assume
  $\mathcal{M}(t)=\mathcal{M}(t+1)=\mathcal{M}(t+2)=\emptyset$. We can
  also assume
  \begin{eqnarray*}
    \left|x_{\hat{l}(t)}(t)-x_{\hat{l}(t+1)}(t+1)\right|&<&\frac{\varepsilon\alpha\beta^2}{12},\\
    \left|x_{\hat{l}(t+1)}(t+1)-x_{\hat{l}(t+2)}(t+2)\right|&<&\frac{\varepsilon\alpha\beta^2}{12},\text{ and}\\
    d\left(\hat{l}(t+1),0,t+1\right)&>&\varepsilon+\frac{\varepsilon\alpha\beta}{12}
  \end{eqnarray*}
  since otherwise we have a good iteration in at most $2$ time
  steps. At first we claim $\mathcal{N}(t+1)\cap K=\emptyset$.  If at
  time $t$ there is a truth seeker $i\in\mathcal{N}(t)\cap K$ then we
  have
  \begin{eqnarray*}
    x_i(t+1)&\ge& x_{\hat{l}(t)}(t)+\alpha\varepsilon-\frac{(1-\alpha)(1-\beta)\varepsilon\alpha\beta}{12}\\
    &\ge& x_{\hat{l}(t)}(t) +\frac{\varepsilon\alpha\beta^2}{12}+\frac{\varepsilon\alpha\beta}{12}\\
    &\ge& x_{\hat{l}(t+1)}(t+1) +\frac{\varepsilon\alpha\beta}{12}.
  \end{eqnarray*}
  So the only truth seekers that have a chance to move into the set
  $\mathcal{N}(t+1)$ could be those of the set $\mathcal{F}(t)$. So
  let truth seeker $i$ be in the set $\mathcal{F}(t)\cap K$, with
  $x_i(t)=x_{\hat{l}(t)}(t)+\delta$, where
  $\varepsilon<\delta<\varepsilon+\frac{\varepsilon\alpha\beta}{12}$.
  (Truth seekers where $\delta\ge
  \varepsilon+\frac{\varepsilon\alpha\beta}{12}$ are ruled out by
  Lemma \ref{lemma_interval_2}.) We have
  \begin{eqnarray*}
     x_i(t+1) &\ge& \underset{\le x_i(t)}{\underbrace{x_{\hat{l}(t)}(t)+\varepsilon}}-(1-\alpha)(1-2\beta)\\
              &\ge& x_{\hat{l}(t)}(t)+\varepsilon\alpha\\
              &\ge& x_{\hat{l}(t)}(t)+\frac{\varepsilon\alpha\beta^2}{12}+\frac{\varepsilon\alpha\beta}{12}\\
              &\ge& x_{\hat{l}(t+1)}(t+1)+\frac{\varepsilon\alpha\beta}{12}.
  \end{eqnarray*}
  Similarly, we can deduce $\mathcal{N}(t+2)\cap K=\emptyset$. Now we
  can assume that the individuals of $\mathcal{N}(t+1)$, who are all
  ignorants, are in the hope interval at time $t+1$, since otherwise
  we would have a good iteration after $1$ time step. So there exist
  individuals $i\in\mathcal{N}(t+1)$ and $j\in\mathcal{F}(t+1)$ with
  $\left|x_i(t+1)-x_j(t+1)\right|\le\varepsilon$. We set
  $x_i(t+1)=x_{\hat{l}(t+1)}(t+1)+\delta$, where $0\le
  \delta\le\frac{\varepsilon\alpha\beta}{12}$ and calculate
  \begin{eqnarray*}
    x_i(t+2) &\ge& x_i(t+1)-(1-2\beta)\delta+
    \underset{i}{\underbrace{\beta\cdot 0}}
    +\underset{j}{\underbrace{\beta\left(\varepsilon-\frac{\varepsilon\alpha\beta}{12}\right)}}\\
    &\ge& x_{\hat{l}(t+1)}(t+1)+\frac{\varepsilon\alpha\beta^2}{12}+\frac{\varepsilon\alpha\beta}{12}\\
    &\ge& x_{\hat{l}(t+2)}(t+2)+\frac{\varepsilon\alpha\beta}{12}.
  \end{eqnarray*}
  For the other direction we have
  \begin{eqnarray*}
    x_i(t+2) &\le& x_i(t+1)-\underset{\hat{l}(t+1)}{\underbrace{\beta\delta}}+(1-2\beta)\varepsilon\\
             &\le& x_{\hat{l}(t+1)}(t+1)+\frac{\varepsilon\alpha\beta}{12}+\varepsilon-2\beta\varepsilon\\
             &\le& x_{\hat{l}(t+1)}(t+1)+\frac{\varepsilon\alpha\beta^2}{12}+\varepsilon\\
             &\le& x_{\hat{l}(t+2)}(t+2)+\varepsilon.
  \end{eqnarray*}
  Thus, $i\in\mathcal{M}(t+2)$, which results in a good iteration in three time steps.
\end{proof}

Thus, we can conclude:
\begin{corollary}
  After a finite number $T_1(\varepsilon,n,\alpha,\beta)$ of steps we
  have $x_{\hat{l}(T_1)}(T_1)\ge h-
  \varepsilon-\frac{\varepsilon\alpha\beta}{12}$ and
  $x_{\hat{u}(T_1)}(T_1)\le
  h+\varepsilon+\frac{\varepsilon\alpha\beta}{12}$.
\end{corollary}

Due to Lemma \ref{lemma_lonely} there can not exist a general bound on
the convergence that does not depend on $\alpha$. We consider the two
side lengths $\ell_2(t):=|x_{\hat{u}(t)}(t)-h|$ and
$\ell_1(t):=|x_{\hat{l}(t)}(t)-h|$ of the hope interval. Clearly $\ell_1(t)$
and $\ell_2(t)$ are not increasing due to Lemma
\ref{lemma_decreasing}. For $t\ge T_1$ we have
$\ell_1(t),\ell_2(t)\le\varepsilon+\frac{\varepsilon\alpha\beta}{12}$

\begin{lemma}
  \label{lemma_epsilon_interval}
  If $\ell_1(t)+\ell_2(t)\le\varepsilon$ then we have
  $$
    (\ell_1(t+2)+\ell_2(t+2))\le (\ell_1(t)+\ell_2(t))\cdot\left(1-\frac{\alpha\beta}{2}\right).
  $$
\end{lemma}
\begin{proof}
  Let us assume, without loss of generality, that $\ell_1(t)\ge
  \ell_2(t)$. At first we consider the case $\ell_2(t)>0$. If $i$ is an
  ignorant with $x_i(t)=h-\ell_1(t)+\delta$ then we have
  \begin{eqnarray*}
    x_i(t+1) &\ge& h-\ell_1(t)+\delta -(1-2\beta)\delta+\beta(\ell_1(t)+\ell_2(t)-\delta)\\
    &\ge& h-(1-\beta)\ell_1(t).
  \end{eqnarray*}
  For a truth seeker $i$ with $x_i(t)=h-\ell_1(t)+\delta$ we have
  \begin{eqnarray*}
    x_i(t+1) &\ge& h-\ell_1(t)+\delta-\alpha(\delta-\ell_1(t))-(1-\alpha)(1-2\beta)\delta+\\
    &&(1-\alpha)\beta(\ell_1(t)+\ell_2(t)-\delta)\\
    &\ge&h-\ell_1(t)+\beta\delta(1-\alpha)+\alpha \ell_1(t)(1-\beta)+\beta \ell_2(t)(1-\alpha)+\beta \ell_1(t)\\
    &\ge&h-(1-\beta)\ell_1(t).
  \end{eqnarray*}
  Similarly we obtain $x_i(t+1)\le h+(1-\beta)\ell_2(t)$ in both cases.

  \medskip

  Next we consider the case $\ell_1(t)>\ell_2(t)=0$ and
  $l_i(t+1)>l_i(t)\cdot\left(1-\frac{\alpha}{2}\right)$.  Let $i$ be
  an arbitrary truth seeker with opinion $x_i(t)=h-\ell_1(t)+\delta$.  We
  have
  \begin{eqnarray*}
    x_i(t+1)&\ge& h-\ell_1(t)+\delta+\alpha(\ell_1(t)-\delta)-(1-\alpha)(1-\beta)\delta\\
    &\ge & h-\ell_1(t)+\alpha \ell_1(t).
  \end{eqnarray*}
  Thus, we have $x_i(t+1)\ge h-\ell_1(t+1)+\frac{\alpha}{2}\cdot
  \ell_1(t)$. If $j$ is an ignorant with $x_j(t+1)=h-\ell_1(t+1)+\delta$,
  then we have
  \begin{eqnarray*}
    x_j(t+2)&\ge& h-\ell_1(t+1)+\delta-(1-2\beta)\delta+\beta\left(\frac{\alpha}{2}\cdot \ell_1(t)-\delta\right)\\
    &\ge& h-\ell_1(t)+\frac{\alpha\beta \ell_1(t)}{2}.
  \end{eqnarray*}
  For an arbitrary truth seeker $j$ we have
  \begin{eqnarray*}
    x_i(t+2)&\ge& h-\ell_1(t+1)+\alpha \ell_1(t+1)\\
    &\ge& h-\ell_1(t)+\frac{\alpha\beta \ell_1(t)}{2}.
  \end{eqnarray*}

  \medskip

  \noindent
  Thus, in all cases we have $(\ell_1(t+2)+\ell_2(t+2))\le
  (\ell_1(t)+\ell_2(t))\cdot\left(1-\frac{\alpha\beta}{2}\right)$.
\end{proof}

This states that once the length of the hope interval becomes at most
$\varepsilon$ its lengths converges to zero.

\begin{lemma}
  \label{lemma_one_step}
  Let $t\ge T_1$.  If there exists an individual~$i$ with
  $\frac{\alpha\beta \ell_1(t)}{12}\le d(\hat{l}(t),i,t)\le\varepsilon$,
  then we have $\ell_1(t+1)\le \ell_1(t)\cdot
  \left(1-\frac{\alpha\beta^2}{12}\right)$. If there exists an
  individual $i$ with $\frac{\alpha\beta \ell_2(t)}{12}\le
  d(\hat{u}(t),i,t)\le\varepsilon$, then we have $\ell_2(t+1)\le
  \ell_2(t)\cdot \left(1-\frac{\alpha\beta^2}{12}\right)$.
\end{lemma}
\begin{proof}
  Due to symmetry it suffices to prove the first statement.  Let $j$
  be an ignorant with $x_j(t)=h-\ell_1(t)+\delta$, where $\delta\ge
  0$. We have
  \begin{eqnarray*}
     x_j(t+1) &\ge& h-\ell_1(t)+\delta-(1-2\beta)\delta+\underset{i}{\underbrace{\beta\left(
     \frac{\alpha\beta \ell_1(t)}{12}-\delta\right)}}\\
     &\ge& h-\left(1-\frac{\alpha\beta^2}{12}\right)\ell_1(t).
  \end{eqnarray*}
  For a truth seeker $j$ with $x_j(t)=h-\ell_1(t)+\delta$, $\delta\ge 0$
  we have
  \begin{eqnarray*}
    x_j(t+1) &\ge& h-\ell_1(t)+\delta+\alpha(\ell_1(t)-\delta)-(1-\alpha)(1-2\beta)\delta+\\
    &&\underset{i}{\underbrace{(1-\alpha)\beta\left(\frac{\alpha\beta \ell_1(t)}{12}-\delta\right)}}\\
    &\ge& h-\ell_1(t)+\beta\delta(1-\alpha)+\alpha \ell_1(t)\left(1-\frac{\alpha\beta^2}{12}\right)+\frac{\alpha\beta^2
    \ell_1(t)}{12}\\
    &\ge& h-\left(1-\frac{\alpha\beta^2}{12}\right)\ell_1(t).
  \end{eqnarray*}
\end{proof}

For transparency we introduce the following six sets:
\begin{eqnarray*}
  \mathcal{N}_1(t):&=&\left\{i\in[n]\mid d(\hat{l}(t),i,t)
  <\frac{\alpha\beta \ell_1(t)}{12}\right\},\\
  \mathcal{N}_2(t):&=&\left\{i\in[n]\mid d(\hat{u}(t),i,t)
  <\frac{\alpha\beta \ell_2(t)}{12}\right\},\\
  \mathcal{M}_1(t):&=&\left\{i\in[n]\mid \frac{\alpha\beta \ell_1(t)}{12}\le d(\hat{l}(t),i,t)\le\varepsilon\right\},\\
  \mathcal{M}_2(t):&=&\left\{i\in[n]\mid \frac{\alpha\beta \ell_2(t)}{12}\le d(\hat{u}(t),i,t)\le\varepsilon\right\},\\
  \mathcal{F}_1(t):&=&\left\{i\in[n]\mid d(\hat{l}(t),i,t)>\varepsilon,\,x_i(t)\le h+\ell_2(t)\right\},\\
  \mathcal{F}_2(t):&=&\left\{i\in[n]\mid d(\hat{u}(t),i,t)>\varepsilon,\,x_i(t)\ge h-\ell_1(t)\right\}.\\
\end{eqnarray*}

With this the individuals of the hope interval are partitioned into
$$
  \mathcal{N}_1(t)\cup\mathcal{M}_1(t)\cup\mathcal{F}_1(t)=\mathcal{N}_2(t)\cup\mathcal{M}_2(t)\cup\mathcal{F}_2(t).
$$

\begin{lemma}
  \label{lemma_two_step}
  If for $k\in\{1,2\}$ and $t\ge T_1$ there exists an ignorant $i\in
  \mathcal{N}_k(t)$ and an individual $j\in\mathcal{F}_k(t)$ with
  $|x_i(t)-x_j(t)|\le\varepsilon$ then $l_k(t+2)\le l_k(t)\cdot
  \left(1-\frac{\alpha\beta^2}{12}\right)$.
\end{lemma}
\begin{proof}
  If $l_k(t+1)>l_k(t)\cdot \left(1-\frac{\alpha\beta^2}{12}\right)$,
  then it is easy to check that the influence of Individual~$j$
  suffices to put ignorant $i$ in set $\mathcal{M}_k(t+1)$. In this
  case we can apply Lemma \ref{lemma_one_step}
\end{proof}

\begin{lemma}
  \label{lemma_no_near_truth_seeker}
  If $\mathcal{N}_k(t+1)\cap K\neq\emptyset$ and $t\ge T_1$ then
  $l_k(t+1)\le l_k(t)\cdot\left(1-\frac{\alpha}{2}\right)$ for
  $k\in\{1,2\}$.
\end{lemma}
\begin{proof}
  Due to symmetry it suffices to consider $k=1$. So let $i$ be a truth
  seeker with $i\in\mathcal{N}_1(t+1)$.  We set
  $x_i(t)=h-\ell_1(t)+\delta$ and calculate
  \begin{eqnarray*}
    x_i(t+1) &\ge& h-\ell_1(t)+\delta +\alpha(\ell_1(t)-\delta)-(1-\alpha)(1-\beta)\delta\\
    &\ge& h-(1-\alpha)\ell_1(t).
  \end{eqnarray*}
\end{proof}

\begin{lemma}
  \label{lemma_one_shrinks}
  For $t\ge T_1$ we have $l_k(t+3)\le l_k(t)\cdot
  \left(1-\frac{\alpha\beta^2}{12}\right)$ for at least one
  $k\in\{1,2\}$.
\end{lemma}
\begin{proof}
  Due to Lemma \ref{lemma_no_near_truth_seeker} we can assume
  $\mathcal{N}_k(t+1)\cap K=\emptyset$. At time $t+1$ there must be a
  truthseeker $i$. Without loss of generality, we assume $x_i(t)\le h$
  and $i=\tilde{l}(t+1)$. Due to Lemma \ref{lemma_one_step} we can
  assume $i\in\mathcal{F}_1(t+1)$. Now let $j_1$ be the ignorant with
  smallest opinion fulfilling $d(i,j_1,t+1)\le\varepsilon$. If
  $j_1\in\mathcal{N}_1(t+1)$ then we can apply Lemma
  \ref{lemma_two_step} with $j_1$ and $i$. Otherwise we let $j_2$ be
  the ignorant with smallest opinion fulfilling
  $d(j_1,j_2,t+1)\le\varepsilon$. So we have
  $d(j_2,i,t+1)>\varepsilon$ and $j_2\in\mathcal{N}_1(t+1)$. Thus, we
  can apply Lemma \ref{lemma_two_step} with $j_2$ and $j_1$.
\end{proof}

\begin{lemma}
  \label{lemma_greater_epsilon}
  If $l_k(t)>\varepsilon$ and $t\ge T_1$ then we have
  $l_k(t+3)-\varepsilon\le (l_k(t)-\varepsilon)\cdot
  \left(1-\frac{\alpha\beta^2}{12}\right)$ or
  $l_k(t+3)\le\varepsilon$.
\end{lemma}
\begin{proof}
  Due to Lemma \ref{lemma_no_near_truth_seeker}, we can assume
  $\mathcal{N}_k(t+1)\cap K=\emptyset$ and, due to Lemma
  \ref{lemma_one_step}, we can assume
  $\mathcal{M}_k(t+1)=\emptyset$. Due to symmetry, we only consider the
  case $k=1$.  Let $i\in\mathcal{N}_1(t+1)$ the ignorant with largest
  opinion $x_i(t+1)$, meaning that $d(\hat{l}(t+1),i,t+1)$ is
  maximal. 

  If there exists an individual $j\in\mathcal{F}_1(t+1)$ with
  $d(i,j,t+1)\le\varepsilon$, then we can apply Lemma
  \ref{lemma_two_step}. If no such individual $j$ exists then we must
  have $d(i,0,t+1)\le\varepsilon$ or $\ell_1(t+1)=0$. So only the first
  case remains. We set $\delta=d(\hat{l}(t+1),i,t+1)\ge
  \varepsilon-\ell_1(t+1)$. Let $h\in\mathcal{N}_1(t+1)$ be an ignorant
  with $x_h(t+1)=x_{\hat{l}(t+1)}(t+1)+\mu$, where
  $0\le\mu\le\delta$. For time $t+2$ we get
  \begin{eqnarray*}
    x_h(t+2) &\ge& x_{\hat{l}(t+1)}(t+1)+\mu -(1-2\beta)\mu+\beta(\delta-\mu)\\
             &\ge& x_{\hat{l}(t+1)}(t+1)+\beta\delta\\
             &\ge& x_{\hat{l}(t+1)}(t+1)+\beta(\varepsilon-\ell_1(t+1)).
  \end{eqnarray*}
\end{proof}

From Lemma \ref{lemma_one_shrinks} and Lemma
\ref{lemma_greater_epsilon} we conclude:

\begin{corollary}
  \label{lemma_everything_small}
  There exists a finite number $T_2(\varepsilon,n,\alpha,\beta)$ so that we have
  $$
    \ell_1(t)+\ell_2(t)\le\varepsilon+\frac{\varepsilon\alpha^2\beta^3}{60},\quad\text{and}\quad
    \min(\ell_1(t),\ell_2(t))\le\frac{\varepsilon\alpha^2\beta^3}{60}
  $$
  for all $t\ge T_2(\varepsilon,n,\alpha,\beta)$.
\end{corollary}
We would like to remark that, e.g.{},
$T_2(\varepsilon,n,\alpha,\beta)=T_1(\varepsilon,n,\alpha,\beta)+\frac{36}{\alpha^2\beta^4}$
suffices.

\begin{lemma}
  \label{lemma_final}
  For each $t\ge T_2(\varepsilon,n,\alpha,\beta)$ we have
  $$
    \ell_1(t+3)+\ell_2(t+3)\le\varepsilon
  $$
  or
  $$
    d(k,0,t)\le\frac{\varepsilon\alpha^2\beta^3}{60}\cdot\left(1-\frac{\alpha\beta}{2}\right)^{\left\lfloor\frac{t-T_2}{2}\right\rfloor}
  $$
  for all truth seekers $k\in K$.
\end{lemma}
\begin{proof}
  Without loss of generality, we assume
  $\ell_1(t)\le\frac{\varepsilon\alpha^2\beta^3}{60}$ and prove the
  statement by induction on $t$. Due to Lemma \ref{lemma_one_step} and
  Lemma \ref{lemma_no_near_truth_seeker}, we can assume
  $\mathcal{M}_2(t+r)=\mathcal{N}_2(t+r)\cap K=\emptyset$ for
  $r\in\{0,1\}$ since otherwise we would have
  $\ell_1(t+3)+\ell_2(t+3)\le\varepsilon$. Thus, we have
  $d(k,0,t+r)\le\frac{\varepsilon\alpha^2\beta^3}{60}$ for all $k\in
  K$ and $K\subseteq \mathcal{F}_2(t+r)$.

  Due to Lemma \ref{lemma_one_step} for $r\in\{0,1\}$, the individuals
  in $\mathcal{F}_2(t+r)$ are not influenced by the individuals in
  $\mathcal{F}_2(t+r)$ since otherwise we would have
  $\ell_1(t+3)+\ell_2(t+3)\le\varepsilon$. Thus, we can apply Lemma
  \ref{lemma_epsilon_interval} for the individuals in
  $\mathcal{F}_2(t)$.
\end{proof}


From the previous lemmas we can conclude Theorem \ref{main_result} and
Theorem \ref{thm_interupted_convergent}.
\begin{proof}(Proof of Theorems~\ref{main_result} and~\ref{thm_interupted_convergent}.)
  After a finite time $T_2(\varepsilon,n,\alpha,\beta)$ we are in a
  \textit{nice} situation as described in Lemma
  \ref{lemma_everything_small}. If we have
  $\ell_1(T_2+3)+\ell_2(T_2+3)\le\varepsilon$ then we have an ordinary
  convergence of the truth seekers being described in Lemma
  \ref{lemma_epsilon_interval}. Otherwise we have
  $d(k,0,T_2)\le\frac{\varepsilon\alpha^2\beta^3}{60}$ for all truth
  seekers $k\in K$. Due to Lemma \ref{lemma_final} and Lemma
  \ref{lemma_epsilon_interval} either we have
$$
d(k,0,t)\le\frac{\varepsilon\alpha^2\beta^3}{60}\cdot\left(1-\frac{\alpha\beta}{2}\right)^{\left\lfloor\frac{t-T_2}{2}\right\rfloor}
$$
for all truth seekers $k\in K$ and all $t\ge T_2$, or there exists an
$S\in\mathbb{N}$, such that we have
\begin{enumerate}
\item[(1)]
  $d(k,0,t)\le\frac{\varepsilon\alpha^2\beta^3}{60}\cdot\left(1-\frac{\alpha\beta}{2}\right)^{\left\lfloor\frac{t-T_2}{2}\right\rfloor}$
  for all $T_2\le t\le S$,
\item[(2)]
  $d(k,0,t)\le\varepsilon\left(1-\frac{\alpha\beta}{2}\right)^{\left\lfloor\frac{t-S-3}{2}\right\rfloor}$
  for all $t\ge S+3$,
\end{enumerate}
for all $k\in K$.  The latter case is $1$-fold interrupted
convergence. Thus, the Hegselmann-Krause Conjecture is proven.
\end{proof}

\section{Remarks}

In this section we would like to generalize the Hegselmann-Krause
Conjecture and show up which requirements can not be weakened.

\begin{lemma}
  A finite number $n$ of individuals and symmetric confidence
  intervals are necessary for a convergence of the truth seekers.
\end{lemma}
\begin{proof}
  Infinitely many ignorants can clearly hinder a truth seeker in
  converging to the truth. If the confidence intervals are not
  symmetric then it is easy to design a situation where some ignorants
  are influencing a truth seeker which does not influence the
  ignorants, so that the truth seeker has no chance to converge to the
  truth.
\end{proof}

\begin{lemma}
  The condition $\beta_{ij}(t)\ge\beta>0$ is necessary for a
  convergence of the truth seekers.
\end{lemma}
\begin{proof}
  If we only require $\beta_{ij}(t)>0$, then we have the following
  example: $n=2$, $x_1(0)=1-\frac{1}{5}\varepsilon$,
  $x_2(0)=1-\varepsilon$, $\alpha_1(t)=\frac{1}{5}$, $\alpha_2(t)=0$,
  $\beta_{11}(t)=\left(\frac{1}{2}\right)^{t+1}$,
  $\beta_{12}(t)=1-\left(\frac{1}{2}\right)^{t+1}$,
  $\beta_{21}(t)=\left(\frac{1}{2}\right)^{t+1}$,
  $\beta_{22}(t)=1-\left(\frac{1}{2}\right)^{t+1}$, and $h=1$. By a
  straight forward calculation we find that
  $|x_1(t)-h|\ge\frac{1}{2}\varepsilon$ for $t\ge 1$.
\end{proof}

We remark that conditions like $\beta_{ij}(t)+\beta_{ij}(t+1)\ge
2\beta$ would also not force a convergence of the truth seekers in
general. One might consider an example consisting of two ignorants
with starting positions $h\pm\frac{7}{10}\varepsilon$ and a truth
seeker $k$ with starting position $h-\frac{1}{5}\varepsilon$. We may
choose suitable $\beta_{ij}(t)$ and $\alpha_i(t)$ so that we have
$|h-x_k(t)|\ge\frac{1}{5}\varepsilon$ for all $t$,
$h-x_k(t)\ge\frac{1}{5}\varepsilon$ for even $t$ and
$x_k(t)-h\ge\frac{1}{5}\varepsilon$ for odd $t$.

\medskip

For the next lemma we need a generalization of Definition
\ref{def_interupted_convergent}.

\begin{definition}
  Given $\varepsilon$, $\alpha$, $\beta$, $n$, we say that the truth
  seekers $k \in K$ are $r$-fold interrupted convergent, if for each
  $\gamma > 0$ there exists $r+1$ functions
  $T_i^s(\gamma,\varepsilon,\alpha,\beta,n,T_{i-1}^e)$,
  $i=1,\dots,r+1$, so that for each (WASBOCOS) $\Omega$ with structural
  parameters $\varepsilon$, $\alpha$, $\beta$ and $n$ there exist
  $T_i^e\in\mathbb{N}$, $i=1,\dots,r$ satisfying
  $$
    \forall k\in K,\,\forall t\in[T_i^s(\gamma,\varepsilon,\alpha,\beta,n,T_{i-1}^s),T_i^e]:\, |x_k(t)-h|<\gamma
  $$
  for $i=1,\dots,r$, where $T_0^e=0$, and 
  $$
    \forall k\in K,\,\forall t\ge T_{r+1}^s(\gamma,\varepsilon,\alpha,\beta,n,T_r^e):\,|x_k(t)-h|<\gamma.
  $$
\end{definition}

\begin{lemma}
  The condition $\alpha_i(t)=0$ for all $i\in\overline{K}$ is
  necessary for Theorem \ref{thm_interupted_convergent}.  If it is
  dropped then the truth seekers are not ($|\overline{K}|-1$)-fold
  convergent in general.
\end{lemma}
\begin{proof}
  At first we remark that it clearly suffices to have $\alpha_i(t)=0$
  for all $i\in\overline{K}$ only for all $t\ge T$, where $T$ is a fix
  integer. W.l.o.g.{} we assume $T=0$ and consider the following
  example: $h=1$, $x_i(0)=1-2i\varepsilon$, $1\in K$, $1\neq
  i\in\overline{K}$, $\beta_{ij}(t)=\beta$, $\alpha_i(t)=\alpha$ for
  the truth seekers, and $\alpha_i(t)=0$ for the ignorants until we
  say otherwise. Let there be a given $\gamma>0$ being sufficiently
  small. There exists a time $T_1$ until $x_1(T_1)<1-\gamma$. Up to
  this time no other individual has changed its opinion. After time
  $T_1+1$ we suitably choose $\alpha_2(t)$ so that we have
  $\frac{1}{2}\varepsilon\le
  x_1(\tilde{T}_1)-x_2(\tilde{T}_1)\le\varepsilon$. So at time
  $\tilde{T}_1+1$ the convergence of truth seeker $1$ is interrupted
  the first time. After that we may arrange it that $x_1$ and $x_2$
  get an equal opinion and will never differ in there opinion in the
  future. Now there exists a time $T_2$ until
  $x_2(T_2)=x_1(T_2)<1-\gamma$ and we may apply our construction
  described above again. Thus, every ignorant $i\in\overline{K}$ may
  cause an interruption of the convergence of the truth seekers.
\end{proof}

\begin{conjecture}
  If we drop the condition $\alpha_i(t)=0$ for all $i\in\overline{K}$
  in Theorem \ref{thm_interupted_convergent} then we have
  ($|\overline{K}|$)-fold convergence of the truth seekers.
\end{conjecture}

The Hegselmann-Krause Conjecture might be generalized to opinions in
$\mathbb{R}^m$ instead of $\mathbb{R}$ when we use a norm instead of
$|\cdot|$ in the definition of the update formula. Using our approach
to prove this $m$-dimensional conjecture would become very technical,
so new ideas and tools are needed. We give an even stronger
conjecture:

\begin{conjecture}
  The $m$-dimensional generalized Hegselmann-Krause Conjecture holds
  and there exists a function $\phi(\Omega,\gamma)$ so that the truth
  seekers in an arbitrary generalized (WASBOCOS) $\Omega$ are
  $\phi(\Omega,\gamma)$-fold interrupted convergent in $\varepsilon$,
  $\alpha$, $\beta$, and $n$.
\end{conjecture}



\end{document}